\documentclass[10pt]{article}
\usepackage{amsmath, amsfonts, amsthm, amssymb, verbatim}
\usepackage{graphicx}
\usepackage[mathcal]{euscript}
\usepackage{amstext}
\usepackage{amssymb,mathrsfs}
\usepackage[latin1]{inputenc}
\newcommand \be     {\begin{equation}}
\newcommand \ee     {\end{equation}}

\newcommand \del     \partial
\DeclareMathOperator  \sgn {sgn}

\newcommand{\ess}{\,{\rm ess}}

\newcommand{\td}[1]{{\tilde{#1}}}

\def\XXint#1#2#3{{\setbox0=\hbox{$#1{#2#3}{\int}$}
\vcenter{\hbox{$#2#3$}}\kern-.5\wd0}}

\DeclareMathOperator\dive {div}
 \newcommand{\R}{\mathbb R}

\newcommand{\ic}{\,\triangledown\,}
\newcommand{\gc}{\,\square\,}

\newcommand{\Lip}{\,{\rm Lip}}

\newtheorem{theorem}{Theorem}[section]

 \newtheorem{remark}[theorem]{Remark}
\newtheorem{lemma}[theorem]{Lemma}

\newtheorem{definition}[theorem]{Definition}

 \def\beqs{\begin{eqnarray*}}
 \def\enqs{\end{eqnarray*}}
 \def\beq{\begin{eqnarray}}
 \def\enq{\end{eqnarray}}

\begin{document}
\title{Multi-Time Systems of Conservation Laws}
\author{Aldo Bazan$^1$, Paola Loreti$^2$, Wladimir Neves$^1$}

\date{}

\maketitle

\footnotetext[1]{ Instituto de Matem\'atica, Universidade Federal
do Rio de Janeiro, C.P. 68530, Cidade Universit\'aria 21945-970,
Rio de Janeiro, Brazil. E-mail: \sl aabp2003@pg.im.ufrj.br,
wladimir@im.ufrj.br}

\footnotetext[2]{ Dipartimento di Scienze di Base e Applicate per
l'Ingegneria, Sapienza Universit\`a di Roma, via A. Scarpa n.16
00161 Roma. E-mail: {\sl paolaloreti@gmail.com.}
\newline
\textit{To appear in:}
}

%
%
\begin{abstract}
Motivated by the work of P.L. Lions and J-C. Rochet \cite{PLLJCR},
concerning multi-time Hamilton-Jacobi equations, we introduce the
theory of multi-time systems of conservation laws. We show the
existence and uniqueness of solution to the Cauchy problem for a
system of multi-time conservation laws with two independent time
variables in one space dimension. Our proof relies on a suitable
generalization of the Lax-Oleinik formula.
\end{abstract}
%

\maketitle


\section{Introduction} \label{IN}

This paper introduces the theory of multi-time systems of
conservation laws. Since to our knowledge nothing is already done
in this direction, we first give the statement of the theory in
Section \ref{ST}. In order to show that the theory is
well-introduced, we prove on the final section the solvability of
the Cauchy problem for a system of multi-time conservation laws
with two independent time variables in one space dimension. The
solvability relies on a generalization of the Lax-Oleinik formula
for two independent times, see Definition \ref{LOF}. Therefore, we
exploit in this paper the explicit Lax formula \eqref{LFMHJE} as
solution for the multi-time Hamilton-Jacobi system \eqref{HJS},
which concept was introduced by Rochet \cite{JCR} in the context
of mathematical economic problems.

\medskip
Besides the philosophical question of the existence of multiple
time dimensions, multi-time phenomena are rather common. For
instance, the time schedule of networks in communication theory,
as well, the traffic models with possibility to consider traffic
jam leading to the use a different time scale. Indeed, processes
that are assumed to start at the same configuration, and the
utility function has to be solution of two different optimization
problems, which coupling is just the initial data. In this
direction, we address the work of Gu, Chung and Hui \cite{GCH},
which is related to traffic flow problems in inhomogeneous
lattices. In fact, traffic flow seems to be one of the prelude
sources of conservation laws, leading for instance to Burgers
equation. Another source of interesting physical problems, where
multi-time phenomena is present, comes from general relativity and
electromagnetism. In this direction, we address the reader to the
works of Neagu and Udriste \cite{MNCU} and Stickforth \cite{JS},
the last one is concerned with the Kepler problem. Although, one
of the most amazing example which leads to multiple dimensions,
even more than two time scales, is given by the string theory, we
address the reader the books of Steven \cite{SSG} and Zwiebach
\cite{BZ}. Most of these physical problems are modelled by systems
of conservation laws, here with two or more time independent
scales. Finally, we have to mention that, one of the motivations
to introduce multi-time conservation laws, comes from the
Lions-Rochet's paper \cite{PLLJCR}, concerning multi-time
Hamilton-Jacobi equations.

\medskip
The mathematical theory of multi-time Hamilton-Jacobi equations
was developed by P.L. Lions and J-C. Rochet \cite{PLLJCR}.  In
that paper Lions and Rochet showed the existence of solution for
\eqref{HJS}. Since then, many works have been done in the context
of multi-time Hamilton-Jacobi equations to extend the results of
Lions and Rochet. The existing literature goes in the direction to
show existence and uniqueness of the solution for more general
class of Hamiltonians and to give weaker  regularity conditions on
the initial data. For instance, Barles and  Tourin \cite{BT} for
Lipschitz initial-data, Plaskacz and Quincampoix \cite{SPMQ} for
initial-data bounded by a semi-continuous function, they present
existence and uniqueness under the hypotheses (H1), (H2), (H3) in
\cite{BT} and Assumption A in \cite{SPMQ}, see Remark \ref{RM}. We
address also the paper of Imbert and Volle \cite{CIMV}, which
consider a more general class of vectorial Hamilton-Jacobi
equations.

\medskip
For our multi-time conservation laws purpose, we were here more
interested in explicit Lipschitz regular solutions for
\eqref{HJS}. Then, under the condition that the initial-data is
Lipschitz and the Hamiltonians are convex and coercive, we give an
explicit and new proof of existence for the multi-time
Hamilton-Jacobi equations, using the the $\inf$-convolution and
$\Gamma$-convolution operations. We show that Lax formula
\eqref{LFMHJE} is a Lipschitz function, which solves the Cauchy
problem \eqref{HJS}, see Theorem \ref{GHLFET}. The same strategy
used to prove Theorem \ref{GHLFET}, with small modifications,
shows also that the Lax formula is a viscosity solution of
\eqref{HJS} in the sense presented on Definition \ref{RLFVS}.
Although the section on viscosity solutions of Hamilton-Jacobi
equations gives known results in literature, here we organize the
topics in order to give the correspondence with multi-time
conservation laws. To make the paper complete on its on, we prefer
to give statements and proofs, adapted to this context. By the
doubling variables technic, we show that there exists at most one
Lipschitz, bounded solution for \eqref{HJS}, see Theorem
\ref{TUHJS}.
Hence the final Section \ref{MTCL1D} presents the existence and
uniqueness solution to the Cauchy problem \eqref{MTSCL1D}. First,
we differentiate the Lax formula with respect to the spacial
variable, and formally show that, it is the best candidate to
solve \eqref{MTSCL1D}. After that we establish in Lemma
\ref{MTLOF} a generalization of the Lax-Oleinik formula for
multi-time variables. Then, we give in Definition \ref{SDMTCL1D}
the exact notion of solution to \eqref{MTSCL1D}, and prove the
existence of an integral solution on Theorem \ref{TMTCLCP}. After
that, by the BV regularity property obtained by the Lax-Oleinik
formula, we show that the integral solution is an entropy solution
to the Cauchy problem \eqref{MTSCL1D} in the sense of Definition
\ref{ENTROPYSOLUTION}. Finally, we prove the uniqueness result on
Theorem \ref{UNQTCL}.

\subsection{Statement of the theory}\label{ST}

The aim of this section is to provide the basic theory for
multi-time systems of conservation laws in multidimensional space
dimensions. We are going to formulate the initial-value problem,
where the systems of equations is complemented by an initial data,
that is, the Cauchy problem.

\medskip
Fix $n$, $d$ and $s$ be positive natural numbers. Let
$t_1,t_2,...,t_n$ be $n$-time independent scales, and consider the
points $(t_1,\ldots,t_n,x_1,\ldots,x_d) \in \R^n \times \R^d$. In
fact, for simplicity of exposition, and without loss of
generality, we consider only two time scales. Moreover, we denote
the spacial variable $(x_1,\ldots,x_d) = x$.

\medskip
Let $U$ be an open subset of $\R^s$, usually called the set of
states, where for each $(t_1,t_2,x)$
$$
  u(t_1,t_2,x) \in U, \qquad \big(u=(u^1,\dots,u^s) \big).
$$
Now, let $f_i: U \to (\R^s)^d$, $(i=1,2)$, be two smooth maps
called flux functions. In general, we postulate that there exist
at most $f_i's$ different flux functions as the number of time
independent variables. Then, we are in position to establish the
following multi-time system of conservation laws in general form
\begin{equation}
\label{SCLMD}
   \begin{aligned}
   \frac{\partial u^i}{\partial t_1} + \frac{\partial
   f_{1j}^i(u)}{\partial x_j}&= 0,
   \\[5pt]
   \frac{\partial u^i}{\partial t_2} + \frac{\partial
   f_{2j}^i(u)}{\partial x_j}&= 0,
   \end{aligned}
\end{equation}
where $(t_1,t_2,x) \in (0,\infty)^2 \times \R^d$, $u(t_1,t_2,x)
\in U$ is the unknown and $f_1$, $f_2$ are given. Moreover, we
remark that the summation convention is used, that is, whenever an
index is repeated once, and only once, a summation over the range
of this index is performed.

\begin{definition}
\label{DHS} The system \eqref{SCLMD} is said to be hyperbolic,
when for any $u \in U$ and any direction $\xi \in S^{d-1}$, each
matrix
$$
  A^i_{1k}:= \frac{\partial f^i_{1j}(u)}{\partial u_k} \; \xi_j
  \quad \text{and} \quad
  A^i_{2k}:= \frac{\partial f^i_{2j}(u)}{\partial u_k} \; \xi_j
  \qquad (1 \leq i,k \leq s),
$$
has $s$ real eigenvalues $\lambda_{i1}(u,\xi) \leq
\lambda_{i2}(u,\xi) \leq \ldots \leq \lambda_{is}(u,\xi)$,
$(i=1,2)$ and is diagonalizable. Therefore, there exist $2 s$
linearly independent right and left corresponding eigenvectors
respectively $r_i(u,\xi)$, $l_i(u,\xi)$, $(i=1,2)$, and
$$
  A_i(u,\xi) \; r_i(u,\xi)= \lambda_{i} \; r_i(u,\xi)
  \quad \text{and} \quad
  l^T_i(u,\xi) \; A_i(u,\xi)= \lambda_{i} \; l_i(u,\xi).
$$
Moreover, when the eigenvalues are all distinct the system
\eqref{SCLMD} is said strictly hyperbolic.
\end{definition}

Hence, we formulate the Cauchy Problem: Find $u(t_1,t_2,x) \in U$
be a function in $(0,\infty)^2 \times \R^d$, which satisfies the
system \eqref{SCLMD} and moreover the initial data
\begin{equation}
\label{ID}
  u(0,0,x)= u_0(x) \qquad \text{for all $x \in \R^d$},
\end{equation}
where $u_0: \R^d \to U$ is a given function.

\medskip Therefore, we have established the Cauchy problem
\eqref{SCLMD}-\eqref{ID} for multi-time systems of conservation
laws in general form and so, many questions are in order at this
point. First of all, one could ask if \eqref{SCLMD}-\eqref{ID} is
well-defined, since this problem seems to be overdetermined.
In this direction, for $d$ and $s$ equals one, Lipschitz
initial-data and smooth convex flux-functions, we show in Section
\ref{MTCL1D} well-posedness to the Cauchy problem
\eqref{SCLMD}-\eqref{ID}.

\medskip
Last but not least, let us write $y= (t_1,t_2,x)$ and for $u(y)
\in \R$, we define
$$
  F(u):=
  \begin{pmatrix}
     u & 0 & f_1(u)
     \\
     0 & u & f_2(u)
  \end{pmatrix}.
$$
Then, from equation \eqref{SCLMD} we have
\begin{equation}
\label{CLSN}
  \dive_y F(u)\equiv \frac{\partial F_{ij}(u)}{\partial y_j}= 0,
  \qquad \Big(i= 1,2; \, j=1,\ldots,d+2 \Big).
\end{equation}
One could expect to apply the standard conservation laws theory.
In this way, we have the following

\begin{definition}
\label{ENTROPYSOLUTION} A field $q(u)$ is called a convex entropy
flux associated with the conservation law \eqref{CLSN}, if there
exists a continuous differentiable convex function $\eta: \R \to
\R$, such that
$$
  q_{ij}(\lambda)= \int_0^\lambda \partial_u \eta(s) \; \partial_u F_{ij}(s)
  \; ds, \qquad \text{for each $\lambda \in \R$}.
$$
Moreover, a measurable and bounded scalar function $u= u(y)$ is
called an entropy solution of the conservation law \eqref{CLSN}
associated with a initial data $u_0 \in L^\infty(\R^d)$, if the
following entropy inequality
$$
  \iint_{\R^{d+2}} q_{ij}(u) \; \partial_{y_j} \phi \; dy \geq 0
$$
holds for each convex entropy flux $q$ and all smooth test
function $\phi$ compactly supported in $(0,T)^2 \times \R^d$, for
all $T>0$, and also the initial data
\begin{equation}
\label{idcles}
  \ess \!\!\!\!\!\! \lim_{t_1,t_2 \to 0^+} \int_\R
|u(t_1,t_2,x) - u_0(x)| \, dx= 0.
\end{equation}
\end{definition}
The main issue of the paper, it will be the existence and
uniqueness result as mentioned before when $s=d=1$. For that, we
exploit the well known idea establish to study conservation laws
(at least in one spatial dimension) from the Hamilton-Jacobi
equations.

\subsection{Functional notation and some results} \label{FN}

Let $ f:  \R^d\to \R\cup \{+\infty\}$. The Legendre-Fenchel
conjugate  of $f$, that is, the function $f^*:\R^d \to \R \cup
\{+\infty\}$ is defined by the formula
$$
  f^*(x):= \sup_{y \in \R^d} \{x \cdot y - f(y) \},
$$
where $x \cdot y$ is the  scalar product of vectors $x,y \in
\R^d$. We recall that, $f^*$ is a convex function, even if $f$ is
not,  and we put $f^{**}= (f^*)^* $. If  $f$ is convex the
Fenchel-Moreau theorem  establishes an important  duality result
between $f$ and its conjugate: if $f$ is  lower semicontinuous and
convex    then $f^{**}=f$. In the following we consider proper
functions.
 If $f: \R^d \to \R$ is coercive, i.e.
$$
  \lim_{\|x\| \to \infty} \frac{f(x)}{\|x\|}= +\infty,
$$
where $\| \cdot \|$ is the Euclidean norm on $\R^d$, then $f^*$ is
also coercive.

For a Lipschitz function $f:\R^d \to \R$, we denote by $\Lip(f)$
the Lipschitz constant of $f$, that is, for each $x,y \in \R^d$,
$$
  |f(x) - f(y)| \leq \Lip(f) \; \|x - y \|.
$$

Given $f,g:\R^d \to \R$, we define (for a more general context,
see Moreau \cite{JJM})
$$
  f \ic g:\R^d \to \R
  \quad \text{and} \quad
  f \gc g:\R^d \to \R,
$$
respectively the infimal-convolution (or $\inf$-convolution ) and
gamma-convolution (or $\Gamma$-convolution) of $f,g$, by
\begin{equation}
\label{INFC}
    \big(f \ic g \big)(x)= \inf_{y \in \R^d} \{f(x-y) + g(y) \}
\end{equation}
and
\begin{equation}
\label{GMMC}
    \big(f \gc g \big)(x)= \big( f^*(x) + g^*(x) \big)^*.
\end{equation}
These operations are dual in the following sense
\begin{theorem}
\label{INFGMMC} Let $f,g:\R^d \to \R$ be two convex functions.
Then,
$$
  f \ic g = f \gc g.
$$
\end{theorem}
The proof could be seen at Rockafellar's book \cite{RTR}, page
145, Theorem 16.4. In fact, there are also more general conditions
on $f$ and $g$, such that these operations are identical, we
address \cite{JJM}.
One recalls further that, infimal-convolution and
gamma-convolution have the properties of commutativity and
associativity.

\medskip
Finally, just for completeness of the paper, let us recall the
Moreau-Yosida approximation, which will be mentioned a posteriori.
For each $\tau>0$, the Moreau-Yosida approximation of $f:\R^d \to
\R$ is given by
$$
  f^\tau(x):=
   \inf_{y \in \R^d} \Big\{ \frac{\|x-y\|^2}{2\, \tau}
  + f(y) \Big\}.
$$

\section{Multi-time Hamilton-Jacobi equations} \label{HJE}

We begin this section by looking to some interesting features of
the multi-time Hamilton-Jacobi equations. For simplicity of
explanation, we consider only two independent times. So, we will
be focus on the following problem: Find $w: (0,\infty)^2 \times
\R^d \to \R$, satisfying
\begin{equation}
\label{HJS}
   \begin{aligned}
      w_{t_1} + H_1(D w) &= 0 \qquad \text{in $(0,\infty)^2 \times
      \R^d$}, \\[5pt]
      w_{t_2} + H_2(D w) &= 0 \qquad \text{in $(0,\infty)^2 \times
      \R^d$}, \\[5pt]
      w(0,0,x) &= g(x) \quad \text{on $\R^d$},
   \end{aligned}
\end{equation}
where $g: \R^d \to \R$ is a given initial datum and $H_i: \R^d \to
\R$ $(i=1,2)$ are given functions usually called Hamiltonians.
Here, we are mostly interested in explicit solutions for
\eqref{HJS} given by formulas with $\R^d$ domains, since they will
be exploited a posteriori in order to show solvability of
multi-time conservation laws.

\medskip
When $t_1=t_2=:t$ and hence $H_1= H_2=:H$, the system \eqref{HJS}
turns to the usual Hamilton-Jacobi equations.  In this context, we
recall some well-known facts and discuss new viewpoints. We
address, for instance, Alvarez, Barron and Ishii \cite{ABI}, Bardi
and Evans \cite{MBLCE}, also Lions and Rochet \cite{PLLJCR}, and
references there in.

\smallskip
1. If $H$ is convex and coercive, $g$ is Lipschitz, then we have
an explicit solution called the Lax formula, that is
\begin{equation}
\label{LFHJE}
  \begin{aligned}
  w_L(t,x)&= \inf_{y \in \R^d} \big\{t H^*\Big(\frac{x-y}{t}\Big)
  + g(y) \big\}
  \\
  &= \inf_{y \in \R^d} \big\{(t H)^*(x-y) + g(y) \big\}
  \\
  &= \big((t H)^* \ic g \big)(x).
  \end{aligned}
\end{equation}
Therefore, the Lax formula is given by the $\inf$-convolution
operation.

\smallskip
2. If $g$ is convex and $H$ is at least continuous, satisfying
\begin{equation}
\label{CFHF}
    \lim_{\|p\| \to \infty} \frac{t H(p) + g^*(p)}{\|p\|} = \infty
\end{equation}
uniformly with respect to any bounded $t$, then we have an
explicit solution called the Hopf formula, that is
\begin{equation}
\label{HJEHF}
  \begin{aligned}
  w_H(t,x)= \big( t H + g^* \big)^*(x),
  \end{aligned}
\end{equation}
which is clearly a convex function.

These two formulas are well-known in the literature as Hopf-Lax
formulas. In fact, there exists a standard habit to call Hopf-Lax
formula undistinguish between them, in spite they are not equal.
For instance, a necessary condition to have both formulas defined,
it is that $H$ and $g$ should be convex (assuming that we have
enough regularity). Moreover, for convex Hamiltonian the Hopf
formula could be written as
$$
  w_H(x)= \big( (t H)^* \gc g \big)(x).
$$
Hence by Theorem \ref{INFGMMC}, we see that
$$
  w_L(x)= \big( (t H)^* \ic g \big)(x)= \big( (t H)^* \gc g
  \big)(x)= w_H(x).
$$
Consequently, $H$ and $g$ be convex are a necessary and sufficient
condition to have $w_L = w_H$, besides that $H$ coercive is
equivalent to condition \eqref{CFHF}.

\bigskip
Now, we turn back our attention to the (vectorial) multi-time
Hamilton-Jacobi problem \eqref{HJS} and, hereafter we do not use
the under scripts $L$ and $H$ respectively to Lax and Hopf
formulas.
Under the assumption that $g$ is convex, continuous on $\R^d$ and
$H_i$ $(i=1,2)$ are continuous and satisfy \eqref{CFHF},
the Proposition 4 at Lions-Rochet's paper \cite{PLLJCR}, presents
an explicit Hopf formula, that is to say
$$
   w(t_1,t_2,x)= \big( t_1 H_1 + t_2 H_2 + g^* \big)^*(x),
$$
which solves \eqref{HJS} $a.e.$ in $[0,T]^2 \times \R^d$, for
$T>0$. Although, they do not present in that paper an explicit Lax
formula. Indeed, considering that $H_i$ $(i=1,2)$ are convex, $g$
is bounded and uniformly continuous, further $Dg$ is measurable
and bounded or $H_i$ $(i=1,2)$ are coercive, they show on
Proposition 5 the following
$$
  w(t_1,t_2,x)= S_{H_1}(t_1) \, S_{H_2}(t_2) g(x)= S_{H_2}(t_2) \, S_{H_1}(t_1) g(x),
$$
which solves \eqref{HJS} $a.e.$ and is Lipschitz on $\R^d \times
[\,\varepsilon,T]^2$ for all $\varepsilon >0$.

\smallskip
On the other hand, following our discussion above, we propose here
to study the following (called) Lax formula, that is
\begin{equation}
\label{LFMHJE}
\begin{aligned}
  w(t_1,t_2,x)&= \big( (t_1 H_1 + t_2 H_2)^* \ic g \big)(x)
\\
  &= \inf_{y \in \R^d} \big\{(t_1 H_1 + t_2 H_2)^*(x-y) + g(y)
  \big\},
\end{aligned}
\end{equation}
where for our purposes, we assume that $g$ is Lipschitz in $\R^d$.

\begin{remark}
\label{RM} Some remarks are in order just now:

1. The regularity of $g$, i.e. Lipschitz continuous, is a natural
assumption in order to show solvability of the multi-time system
of conservation laws. In fact, this condition could be relaxed
using the Moreau-Yosida approximation $g^\tau$ of $g$ and then,
applying the same strategy used in Alvarez, Barron and Ishii
\cite{ABI}.

\smallskip
2. The Lax formula \eqref{LFMHJE} already appears, as well, in
Imbert and Vollet's paper, see \cite{CIMV} to study the vectorial
Hamilton-Jacobi equations. Although, completed different from that
paper, here we are interested to show existence, uniqueness of
\eqref{HJS} and, further Lipschitz regularity of \eqref{LFMHJE} in
an explicit and computationally way, which it will be exploited in
the multi-time conservation laws section.

\smallskip
3. If we agree with the notation $w(t,x)= \big(S_H(t) \; g
\big)(x)$ for \eqref{LFHJE}, then we observe that
$$
  \begin{aligned}
  w(t_1,t_2,x)&= \big( (t_1 H_1 + t_2 H_2 )^* \ic g \big)(x)
  \\
  &= \big( (t_1 H_1)^* \gc (t_2 H_2)^* \ic g \big)(x)
\\
  &= \big( (t_1 H_1)^* \ic (t_2 H_2)^* \ic g \big)(x),
  \end{aligned}
$$
which justifies the notation and commutativity in Proposition 5 at
Lions and Rochet's paper \cite{PLLJCR}.

\smallskip
4. For simplicity, we sometimes denote $
  t_1 H_1 + t_2 H_2 =:\bold{t} \cdot \bold{H}
$
(as obvious notation) and, the Lax formula \eqref{LFMHJE} becomes
$$
  w(t_1,t_2,x)= \big( (\bold{t} \cdot \bold{H})^* \ic g \big)(x).
$$

\smallskip
5. Finally, we give respectively the hypotheses $(H1)-(H3)$ on
Barles and Tourin \cite{BT} and the Assumption A  on Plaskacz and
Quincampoix \cite{SPMQ}:

$(H1)$ For any $R>0$, there exists a constant $K_R>0$, such that
$$
  \begin{aligned}
  |H_i(x,p)| &\leq K_R \qquad \qquad \qquad \text{in $\R^d
  \times \{|p| \leq R \}, i=1,2$},
  \\
  |D_pH_i(x,p)| &\leq K_R \; (1+|x|) \qquad \text{a.e. in $\R^d
  \times \{|p| \leq R \}, i=1,2$}.
  \end{aligned}
$$

$(H2)$ $H_1, H_2$ are coercive uniformly with respect to $x \in
\R^d$.

\smallskip
$(H3)$ $H_1, H_2$ are $C^1$ in $\R^d \times \R^d$ and satisfy
$$
  D_x H_1(x,p) \; D_p H_2(x,p) - D_x H_2(x,p) \; D_p H_1(x,p)= 0,
$$
for each $x,p \in \R^d$. The equality above is always satisfied if
$H_1, H_2$ do not depend on $x$, further the Hamiltonians could be
assumed locally Lipschitz.

\medskip
Assumption $A$; $H(u,p)= \tilde{H}(u,p)+\lambda(u)$, where
$\lambda(u)$ is a $C^1$ real scalar non-negative and
non-increasing function, and $\tilde{H}: \R \times \R^d \to \R$,
satisfy
$$
  \begin{aligned}
  \tilde{H}(u,\cdot) \quad &\text{is a concave and positively
  homogeneous},
  \\
  \tilde{H}(\cdot,p) \quad &\text{is a non-increasing $C^1$
  function}.
  \end{aligned}
$$
\end{remark}

\subsection{Existence}

First, we show that the infimum in \eqref{LFMHJE} is in fact a
minimum, hence the infimal convolution is said exact. Moreover,
$w$ is a continuous function.

\begin{lemma}
\label{MSHJ} Assume that $g:\R^d \to \R$ is a Lipschitz continuous
function, and let $w$ be defined by \eqref{LFMHJE}. Then,
$$
 w(t_1,t_2,x)= \min_{y \in \R^d} \big\{ (\bold{t} \cdot
 \bold{H})^*(x-y) + g(y) \big\}.
$$
Moreover, $w$ is a continuous function.
\end{lemma}

\begin{proof}
By definition of infimum, there exists $\{y_n\}$ on $\R^d$ such
that
$$
  w(t_1,t_2,x)= \lim_{n \to \infty} \big\{ (\bold{t} \cdot
 \bold{H})^*(x-y_n) + g(y_n) \big\}.
$$
If $\{y_n\}$ has at least one convergent subsequence, we are done.
Otherwise, $\{y_n\}$ should be unbounded, which is not the case.
Indeed, recall that $H^{*}_{i}$ ($i=1,2$) are coercive, hence
$(\bold{t} \cdot \bold{H})^*$ is also coercive. Therefore, there
exist $\lambda$ a non-negative real arbitrary number and a
constant $\beta$, such that, for $n$ sufficiently large
$$
  (\bold{t} \cdot \bold{H})^*(x - y_n) \geq \lambda \; \|x-y_{n}\|
  - \beta - 1/n.
$$
Moreover, since the function $g$ is Lipschitz continuous, we have
$$
  g(y_{n})\geq -\Lip(g)\left\|y_{n}\right\| + g(0).
$$
Then, it follows by the above inequalities that
$$
  \begin{aligned}
  (\bold{t} \cdot
 \bold{H})^*(x-y_n) + g(y_n) &\geq \lambda \, \|x - y_n\| - \Lip(g) \, \|y_n\|
  + g(0) - \beta -1/n
  \\[5pt]
  &\geq \lambda \, \big( \|y_n\| - \|x\| \big) - \Lip(g) \, \|y_n\|
  + g(0) - \beta -1/n
  \\[5pt]
  &\geq C \, \|y_n\| + g(0) - \beta -1/n,
  \end{aligned}
$$
where $C$ is a positive constant (take $\lambda > \Lip (g)$).
Then, passing to the limit as $n \to \infty$, we have a
contradiction, since the infimum in \eqref{LFMHJE} is finite.
\end{proof}

\bigskip
The next lemma establish the semigroup property of the Lax
formula.

\begin{lemma}
\label{LSG} Let $g:\R^d \to \R$ be a Lipschitz continuous function
and $w$ defined by \eqref{LFMHJE}. Then, for each $0 \leq s_i <
t_i$, $(i=1,2)$, and all $x \in \R^d$, it follows that
\begin{equation}
\label{SHJ1}
 w(t_1,t_2,x)= \min_{y \in \R^d}
 \big\{ \big( (\bold{t-s}) \cdot \bold{H} \big)^{*}(x-y)
 + w(s_1,s_2,y) \big\}.
\end{equation}
\end{lemma}

\begin{proof}
The proof is a simple application of the $\inf$-convolution and
$\Gamma$-convolution operations. Indeed, we have
$$
  \begin{aligned}
  w(t_1,t_2,x)&= \big( (t_1 H_1 + t_2 H_2 )^* \ic g \big)(x)
\\
  &= \big( ((t_1-s_1) H_1 + s_1 H_1 + (t_2-s_2) H_2 + s_2 H_2 )^* \ic g \big)(x)
\\
  &= \Big(\big((t_1-s_1) H_1 + (t_2-s_2) H_2 \big)^*
  \gc \big(s_1 H_1 + s_2 H_2 \big)^* \ic g \Big)(x)
\\
  &= \Big(\big((t_1-s_1) H_1 + (t_2-s_2) H_2 \big)^*
  \ic \big(s_1 H_1 + s_2 H_2 \big)^* \ic g \Big)(x),
  \end{aligned}
$$
where we have used Theorem \ref{INFGMMC}.
\end{proof}

Now, we prove that $w$ defined by \eqref{LFMHJE} is a Lipschitz
continuous function. Therefore, by Rademacher's Theorem, see
\cite{EG}, differentiable almost everywhere in $\R^d$ and for
almost all $t_1,t_2 >0$.

\begin{lemma}
\label{DSHJ} The function $w$ defined by \eqref{LFMHJE} is
Lipschitz in $[0,\infty)^2 \times \R^d$. Moreover, we have
\begin{equation}
\label{IDMTHJE}
  \lim_{t_{1},t_{2} \to 0} w(t_{1},t_{2},x)= g(x) \quad \text{on $\R^d$}.
\end{equation}
\end{lemma}

\begin{proof}
1. First, fix $t_1,t_2 > 0$ and $x,x_0 \in\R^{d}$. Choose
$y\in\R^{d}$, such that
$$
  w(t_1,t_2,x)=  (t_1 H_1 + t_2 H_2)^*(x-y) + g(y).
$$
Thus we have
$$
  \begin{aligned}
  w(t_{1},t_{2},x_0)-w(t_{1},t_{2},x)&= \min_{z \in \R^d}
  \big\{ (t_1 H_1 + t_2 H_2 )^*(x-z) + g(z) \big\}
  \\[5pt]
  &-(t_1 H_1 + t_2 H_2)^*(x-y) - g(y)
  \\[5pt]
  &\leq g(x_0-x+y) - g(y) \leq \Lip(g) \|x_0 - x\|,
  \end{aligned}
$$
where we have used $z=x_{0}-x+y$. Now, reverting $x_0$ and $x$ in
the above, we obtain
\begin{equation}
\label{LCG}
   | w(t_1,t_2,x) - w(t_1,t_2,x_0) | \leq \Lip(g) \, \|x - x_0 \|,
\end{equation}
that is, $w(t_1,t_2,x)$ is Lipschitz with respect to the spacial
variable $x \in \R^d$.

\medskip
2. Since $g$ is Lipschitz continuous, for each $x,y \in \R^d$, we
have
$$
  g(y) \geq g(x) - \Lip(g) \, \|x - y \|.
$$
Therefore, by definition of $w(t_1,t_2,x)$, we obtain
\begin{equation}
\label{L3LT}
\begin{aligned}
   g(x)-w(t_1,t_2,x) &\leq
    \max_{y \in \R^d} \big\{
    \Lip(g) \, \|x-y\| - \big(t_1 H_1 + t_2 H_2 \big)^*(x-y) \big\}
    \\
    &\leq
    \max_{z \in \R^d} \big\{
    \max_{\xi \in B_{\Lip(g)}(0)} z \cdot \xi - \big(t_1 H_1 + t_2 H_2 \big)^*(z) \big\}
    \\
    &= \max_{\xi \in B_{\Lip(g)}(0)}
     \big(t_1 H_1 + t_2 H_2 \big)(\xi).
\end{aligned}
\end{equation}
On the other hand, taking $x=y$ in the definition of
$w(t_1,t_2,x)$, it follows that
$$
  w(t_1,t_2,x) - g(x) \leq \big(t_1 H_1 + t_2 H_2 \big)^*(0).
$$
Consequently, we obtain
\begin{equation}
\label{L3LT2} \inf_{\xi \in \R^d} \big(\bold{t} \cdot \bold{H}
\big)(\xi)
  \leq g(x) - w(t_1,t_2,x) \leq
  \max_{\xi \in B_{\Lip(g)}(0)}
  \big( \bold{t} \cdot \bold{H} \big)(\xi).
\end{equation}
Furthermore, passing to the limit as $t_1,t_2 \to 0$, we obtain
\eqref{IDMTHJE}.

\medskip
3. Finally, we show that $w$ is Lipschitz continuous with respect
to the time variables. Fix $0 < s_i <t_i$, $(i=1,2)$ and $x \in
\R^d$. By \eqref{LCG} for each $t_1$, $t_2$, we have
$$
  \Lip(w(t_1,t_2,\cdot)) \leq \Lip(g).
$$
Then, we apply the semigroup property of the Lax formula given by
Lemma \ref{LSG} and, moreover proceed similarly as we have done in
step 2 above. Hence the result follows.
\end{proof}

\bigskip
To end up this section, let us show that \eqref{LFMHJE} solves the
multi-time Hamilton-Jacobi partial differential equation in
\eqref{HJS}, wherever $w$ is differentiable. One recalls that, the
initial-data is shown by Lemma \ref{DSHJ}.

\begin{lemma}
\label{TEHJE} Let $(t_1,t_2,x) \in (0,\infty)^2 \times \R^d$ be a
differentiable point for the multitime Lax formula given by
\eqref{LFMHJE}. Then,
$$
   \begin{aligned}
      \partial_{t_1} w(t_1,t_2,x) + H_1\big(D w(t_1,t_2,x)\big) &= 0,\\[5pt]
      \partial_{t_2} w(t_1,t_2,x) + H_2\big(D w(t_1,t_2,x)\big) &= 0.
   \end{aligned}
$$
\end{lemma}

\begin{proof} Let us show the first differential equality, the second is similar.
First, by the semigroup property, we have
\begin{equation}
\label{TEHJ1}
    w(t_1,t_2,x) \leq \big( (t_1 - s_1) H_1 \big)^*(x-y)
    + w(s_1,t_2,y),
\end{equation}
where we have used $0 < s_2= t_2$, $0 < s_1 < t_1$ and $y \in
\R^d$. Take $\delta>0$, $q \in \R^d$ fixed, and replace in
\eqref{TEHJ1}
$
   s_1 \mapsto t_1, \quad t_1 \mapsto t_1 + \delta, \quad
   y \mapsto x \quad \text{and} \quad x \mapsto x+ \delta \,q,
$
thus we have
$$
  w(t_1+\delta, t_2, x+\delta q) - w(t_1, t_2, x)
  \leq \delta \, H_1^*(q).
$$
Then, dividing by $\delta$ and letting to $0^+$, we obtain
$$
   w_{t_1}(t_1,t_2,x) + q \cdot Dw(t_1,t_2,x) -
   H_1^*(q) \leq 0.
$$
Consequently, by the above inequality, it follows that
$$
  \begin{aligned}
   w_{t_1}(t_1,t_2,x) &+ \max_{p \in \R^d} \{p \cdot Dw(t_1,t_2,x) -
   H_1^*\big(p\big)\} \leq 0,
  \end{aligned}
$$
which implies
$$
    \begin{aligned}
   w_{t_1}(t_1,t_2,x) &+ H_1\big(Dw(t_1,t_2,x)\big) \leq 0.
  \end{aligned}
$$

\medskip
Now choose $z \in \R^d$ such that
$$
  w(t_1,t_2,x)= \big( t_1 H_1 + t_2 H_2 \big)^*(x-z) + g(z).
$$
Fix $\delta>0$ and conveniently set $t_1= s_1 + \delta$,
$$
  y= \frac{t_{1}-\delta}{t_{1}} \, x
  + \frac{\delta}{t_{1}} \, z,
  \quad \text{so} \quad \frac{x-z}{t_{1}}= \frac{y-z}{s_1}.
$$
Therefore, by definition of $w(s_1,t_2,y)$, we obtain
$$
  \begin{aligned}
  w(t_{1},t_{2},x)- w(s_1,t_2,y) &\geq
  \big( t_1 H_1 + t_2 H_2 \big)^*(x-z)
  - \big( s_1 H_1 + t_2 H_2 \big)^*(y-z)
\\
  &\geq \delta \; H_1^*\Big(\frac{x-z}{t_1}\Big).
  \end{aligned}
$$
Then, passing to the limit as $\delta \to 0^+$ after divide by
$\delta$, we obtain
\begin{equation}
\label{TEHJ5}
    w_{t_1}(t_1,t_2,x) + \frac{x-z}{t_1} \cdot Dw(t_1,t_2,x) -
    H_1^*\Big(\frac{x-z}{t_1}\Big) \geq 0.
\end{equation}
Finally, we have by \eqref{TEHJ5}
$$
  \begin{aligned}
  w_{t_1}(t_1,t_2,x) &+ H_1\big(Dw(t_1,t_2,x)\big)
  = w_{t_1}(t_1,t_2,x)
  \\[5pt]
  &+ \max_{q \in \R^d} \big\{ q \cdot Dw(t_1,t_2,x) -
  H_1^*(q) \big\}
  \\[5pt]
  &\geq w_{t_1}(t_1,t_2,x)
  \\[5pt]
  &+ \frac{x-z}{t_1} \cdot Dw(t_1,t_2,x) -
    H_1^*\Big(\frac{x-z}{t_1}\Big) \geq 0.
  \end{aligned}
$$
\end{proof}

Consequently, we have proved in this section the following
\begin{theorem}
\label{GHLFET} Let $w$ be the Lax formula given by \eqref{LFMHJE}.
Then, $w$ is Lipschitz continuous, is differentiable a.e. in
$(0,\infty)^2 \times \R^d$, and solves the multi-time
Hamilton-Jacobi initial-value problem
$$
  \begin{aligned}
      w_{t_1} + H_1(D w) &= 0 \qquad \text{a.e. in $(0,\infty)^2 \times
      \R^d$}, \\[5pt]
      w_{t_2} + H_2(D w) &= 0 \qquad \text{a.e. in $(0,\infty)^2 \times
      \R^d$}, \\[5pt]
      w(0,0,x) &= g(x) \quad \text{on $\R^d$}.
   \end{aligned}
$$
\end{theorem}

\begin{definition}
\label{RLFVS} A continuous function $w:(0,\infty)^2 \times \R^d
\to \R$ is called:

\begin {itemize}
\item A viscosity subsolution of the initial-value problem
\eqref{HJS}, provided $$w(0,0,\cdot)= g(\cdot) \quad \text{on
$\R^d$}$$ and for each $\phi\in C^1((0,\infty)^2 \times \R^d)$ if
$w-\phi$ has a local maximum in $(\tau_1,\tau_2,\xi) \in
(0,\infty)^2 \times \R^d$, then
\begin{equation*}
\begin{aligned}
   \phi_{t_1}(\tau_1,\tau_2,\xi) + H_1\big(D \phi(\tau_1,\tau_2,\xi)\big) &\leq 0, \\[5pt]
      \phi_{t_2}(\tau_1,\tau_2,\xi) + H_2\big(D\phi(\tau_1,\tau_2,\xi)\big) &\leq  0.
   \end{aligned}
\end{equation*}
\medskip

\item A viscosity supersolution of the initial-value problem
\eqref{HJS}, provided $$w(0,0,\cdot)= g(\cdot) \quad \text{on
$\R^d$}$$ and for each $\phi\in C^1((0,\infty)^2 \times \R^d)$ if
$w-\phi$ has a local minimum in
$(\tau_1,\tau_2,\xi)\in(0,\infty)^2 \times \R^d$, then
\begin{equation*}
\begin{aligned}
   \phi_{t_1}(\tau_1,\tau_2,\xi) + H_1\big(D \phi(\tau_1,\tau_2,\xi)\big) &\geq 0, \\[5pt]
      \phi_{t_2}(\tau_1,\tau_2,\xi) + H_2\big(D\phi(\tau_1,\tau_2,\xi)\big) &\geq  0.
   \end{aligned}
\end{equation*}
\medskip
\end{itemize}
Moreover, $w$ is said a viscosity solution of \eqref{HJS}, when it
is both a viscosity supersolution and a viscosity subsolution of
\eqref{HJS}.
\end{definition}

One observes that, with a similar strategy used before, it is not
difficult to show that $w$ given by \eqref{LFMHJE} is a viscosity
subsolution and also a viscosity supersolution of \eqref{HJS}.
Then, by definition it is a viscosity solution of \eqref{HJS}.

\subsection{Uniqueness}

In this section using the idea of doubling variables, see for
instance Kruzkov \cite{KRUZKOV}, Crandall, Evans and Lions
\cite{CEL}, we show the uniqueness of bounded Lipschitz solutions
for the initial-value problem \eqref{HJS}.

\begin{theorem} \label{TUHJS} Assume that the initial-data $g$ is a
bounded Lipschitz function, $H_i$ $(i=1,2)$ are convex and
coercive. Then, there exists at most one Lipschitz, bounded
viscosity solution of \eqref{HJS}.
\end{theorem}

\begin{proof}
1. Let $\alpha$ be a positive real number, defined as
\begin{equation}
\label{df} \alpha:=\sup_{[0,+\infty)^{2}\times\R^{d}}
\big(w-\tilde{w} \big),
\end{equation}
where $w$ and $\tilde{w}$ are two Lipschitz, bounded solutions of
\eqref{HJS} with the same initial-data. Now, we choose
$0<\epsilon,\lambda_{1},\lambda_{2}<1$ and define the function
$\Theta$ as
$$
  \begin{aligned}
  \Theta(t_{1},t_{2},s_{1},s_{2},x,y)&:= w(t_{1},t_{2},x)- \tilde{w}(s_{1},s_{2},y)
\\
  &-\rho_{\epsilon,\lambda_{1},\lambda_{2}}(t_{1},t_{2},s_{1},s_{2},x,y)
  \end{aligned}
$$
for each $t_i,s_i \geq 0$  $(i=1,2)$ and $x,y \in \R^d$, where
$$
  \begin{aligned}
  \rho_{\epsilon,\lambda_{1},\lambda_{2}}(t_{1},t_{2},s_{1},s_{2},x,y)&:= \frac{\lambda_{1}}{2}(t_{1}+s_{1}) + \frac{\lambda_{2}}{2}(t_{2}+s_{2})
\\
  &+\epsilon^{-2} \Big((t_{1}-s_{1})^{2}
  +(t_{2}-s_{2})^{2} + \|x-y\|^{2} \Big)
\\
  &+\epsilon \, (\|x\|^{2}+\|y\|^{2})\\
  \end{aligned}
$$
So, as
$$\lim_{\left\|(t_{1},t_{2},s_{1},s_{2},x,y)\right\|\rightarrow+\infty}\rho_{\epsilon,\lambda_{1},\lambda_{2}}(t_{1},t_{2},s_{1},s_{2},x,y)=+\infty,$$
we have
$$\lim_{\left\|(t_{1},t_{2},s_{1},s_{2},x,y)\right\|\rightarrow+\infty}\Theta(t_{1},t_{2},s_{1},s_{2},x,y)=-\infty$$
and, as the function $\Theta$ is continuous in its domain, and it
is proper (not indentically $\pm\infty$), it there must be a point
of maximum, i.e., there exists a point
$(\hat{t}_{1},\hat{t}_{2},\hat{s}_{1},\hat{s}_{2},\hat{x},\hat{y})
\in [0,+\infty)^{4} \times \R^{2d}$, such that
\begin{equation}
\label{mx}
\Theta(\hat{t}_{1},\hat{t}_{2},\hat{s}_{1},\hat{s}_{2},\hat{x},\hat{y})=
\max_{[0,+\infty)^{4} \times \R^d}
\Theta(t_{1},t_{2},s_{1},s_{2},x,y).
\end{equation}
\newline
2. From ($\ref{mx}$), the map

$$(t_{1},t_{2},x)\longmapsto\Theta(t_{1},t_{2},\hat{s_{1}},\hat{s_{2}},x,\hat{y})$$
has a maximum in ($\hat{t}_{1},\hat{t}_{2},\hat{x}$). If we write
$\Theta$ as
$$
  \Theta(t_{1},t_{2},\hat{s_{1}},\hat{s_{2}},x,\hat{y})=w(t_{1},t_{2},x)-v(t_{1},t_{2},x),
$$
where
\begin{equation*}
\begin{aligned}
  v(t_{1},t_{2},x):=
  \tilde{w}(\hat{s}_{1},\hat{s}_{2},\hat{y})
  +\rho_{\epsilon,\lambda_{1},\lambda_{2}}(t_{1},t_{2},\hat{s}_{1},\hat{s}_{2},x,\hat{y})
\end{aligned}
\end{equation*}
then, $(w-v)$ has a maximum in
($\hat{t}_{1},\hat{t}_{2},\hat{x}$). Since $w$ is a viscosity
solution of \eqref{HJS}, it follows that
\begin{equation*}
\begin{aligned}
   v_{t_1}(\hat{t}_1,\hat{t}_2,\hat{x}) + H_1\big(D v(\hat{t}_{1},\hat{t}_{2},\hat{x})\big) &\leq 0, \\[5pt]
      v_{t_2}(\hat{t}_{1},\hat{t}_{2},\hat{x}) + H_2\big(D v(\hat{t}_{1},\hat{t}_{2},\hat{x})\big) &\leq  0.
   \end{aligned}
\end{equation*}
Now, using the definition of $v$ we obtain

\begin{equation}
\label{Eqv}
\begin{aligned}
\frac{\lambda_{1}}{2}+\epsilon^{-2}(\hat{t}_{1}-\hat{s}_{1})+H_{1}\left(\frac{2}{\epsilon^{2}}(\hat{x}-\hat{y})+2\epsilon\hat{x}\right)\leq
0,
\\
\frac{\lambda_{2}}{2}+\epsilon^{-2}(\hat{t}_{2}-\hat{s}_{2})+H_{2}\left(\frac{2}{\epsilon^{2}}(\hat{x}-\hat{y})+2\epsilon\hat{x}\right)\leq
0.
 \end{aligned}
\end{equation}
Analogously, the map
$$(s_{1},s_{2},y)\longmapsto -\Theta(\hat{t}_{1},\hat{t}_{2},s_{1},s_{2},\hat{x},y)$$
has a minimum in ($\hat{s}_{1},\hat{s}_{2},\hat{y}$). Writing
$-\Theta(\hat{t}_{1},\hat{t}_{2},s_{1},s_{2},\hat{x},y)$ as
$$-\Theta(\hat{t}_{1},\hat{t}_{2},s_{1},s_{2},\hat{x},y):=\tilde{w}(s_{1},s_{2},y)-\tilde{v}(s_{1},s_{2},y)$$
hence ($\tilde{w}-\tilde{v}$) has a minimum in
($\hat{s}_{1},\hat{s}_{2},y$), where
$$
  \begin{aligned}
  \tilde{v}(s_{1},s_{2},y):=w(\hat{t}_{1},\hat{t}_{2},\hat{x})-\rho_{\epsilon,\lambda_{1},\lambda_{2}}(\hat{t}_{1},\hat{t}_{2},s_{1},s_{2},\hat{x},y).
  \end{aligned}
$$
Similarly to ($\ref{Eqv}$), we have

\begin{equation}
\label{Eqvv}
\begin{aligned}
-\frac{\lambda_{1}}{2}+\epsilon^{-2}(\hat{t}_{1}-\hat{s}_{1})+H_{1}\left(2\epsilon^{-2}(\hat{x}-\hat{y})
  -2\epsilon\hat{y} \right)\geq 0,
\\
-\frac{\lambda_{2}}{2}+\epsilon^{-2}(\hat{t}_{2}-\hat{s}_{2})+H_{2}\left(2\epsilon^{-2}(\hat{x}-\hat{y})-2\epsilon\hat{y}\right)\geq
0.
\end{aligned}
\end{equation}

3. Finally, making the difference between ($\ref{Eqvv}$) and
($\ref{Eqv}$) with respect to the first line, we have
$$
  \lambda_{1}\leq H_{1}\left(2\epsilon^{-2}(\hat{x}-\hat{y})
  -2\epsilon\hat{y}\right)
  -H_{1}\left(2\epsilon^{-2}(\hat{x}-\hat{y})+2\epsilon\hat{x}\right).
$$
Since $H_{1}$ is locally Lipschitz continuous (and the maximum
point
$(\hat{t_{1}},\hat{t_{2}},\hat{s_{1}},\hat{s_{2}},\hat{x},\hat{y})$
is attained in a compact ball), we have
\begin{equation}
\label{dif} \lambda_{1}\leq
2\epsilon\left\|\hat{y}+\hat{x}\right\|.
\end{equation}
At this point, we need an estimate of
$\left\|\hat{y}+\hat{x}\right\|$ to conclude that $\lambda_1= 0$,
since $\epsilon>0$ is arbitrary. It will be obtained thanks to the
definition of $\rho_{\epsilon,\lambda_{1},\lambda_{2}}$. In fact,
we can fix $0< \epsilon, \lambda_1, \lambda_2 < 1$ so small that
($\ref{df}$) implies
\begin{equation}
\label{alp}
 \Theta(\hat{t}_{1},\hat{t}_{2},\hat{s}_{1},
 \hat{s}_{2},\hat{x},\hat{y})
 \geq
 \sup_{[0,T]^2 \times \R^{2d}} \Theta(t_{1},t_{2},t_{1},t_{2},x,x)
 \geq \frac{\alpha}{2}.
\end{equation}
Moreover, since

$$\Theta(\hat{t}_{1},\hat{t}_{2},\hat{s}_{1},\hat{s}_{2},\hat{x},\hat{y})\geq \Theta(0,0,0,0,0,0),$$
it follows that
$$
  \begin{aligned}
  \rho_{\epsilon,\lambda_{1},\lambda_{2}}(\hat{t}_{1},\hat{t}_{2},\hat{s}_{1},\hat{s}_{2},\hat{x},\hat{y})
  &\leq [w(\hat{t}_{1},\hat{t}_{2},\hat{x}) - w(0,0,0)]
\\
  &-[\tilde{w}(\hat{s}_{1},\hat{s}_{2},\hat{y})-\tilde{w}(0,0,0)].
  \end{aligned}
$$
Since $w$ and $\tilde{w}$ are bounded, we obtain as
$\epsilon\rightarrow 0^{+}$
\begin{equation}
\begin{aligned}
\label{x1}
  |\hat{t}_{1}-\hat{s}_{1}|,|\hat{t}_{2}-\hat{s}_{2}|, \|\hat{x}-\hat{y}\|=
  O(\epsilon),
  \\
  \epsilon \; (\|\hat{x}\|^{2}+\|\hat{y}\|^{2})=O(1).
\end{aligned}
\end{equation}
The last equation of \eqref{x1} implies that
\begin{equation}
\label{x2}
  \begin{aligned}
  \epsilon \, (\|\hat{x}\|+\|\hat{y}\|)&= \epsilon^{\frac{1}{4}}\; \epsilon^{\frac{3}{4}}
  \, (\|\hat{x}\| + \|\hat{y}\|)
  \\
  &\leq\epsilon^{\frac{1}{2}}+ C \, \epsilon^{\frac{3}{2}}(\|\hat{x}\|^{2}+\|\hat{y}\|^{2})
  \leq C \, \epsilon^{\frac{1}{2}}
   \end{aligned}
\end{equation}
for some positive constant $C$. To complete the proof, we use
($\ref{x2}$) in($ \ref{dif}$) and get
$$
  \lambda_{1}\leq 2C\epsilon^{\frac{1}{2}}.
$$
Similarly, we obtain that $\lambda_{2}= 0$, and this contradiction
completes the proof.
\end{proof}

\begin{remark}
Note that in the proof, the points $\hat{t}_{1},\hat{t}_{2},\hat{s}_{1},\hat{s}_{2}$ could be zero, and in that
case, with respect to the time, the function $\Theta$ would be constant.
To see that this does not happen, we recall that
$$\Theta(\hat{t}_{1},\hat{t}_{2},\hat{t}_{1},\hat{t}_{2},\hat{x},\hat{x})
\leq
\Theta(\hat{t}_{1},\hat{t}_{2},\hat{s}_{1},\hat{s}_{2},\hat{x},\hat{y})$$
and from this, we get
$$
\begin{aligned}
  w(\hat{t}_{1},\hat{t}_{2},\hat{x}) &- \tilde{w}(\hat{s}_{1},\hat{s}_{2},\hat{y})-\rho_{\epsilon,\lambda_{1},\lambda_{2}}(\hat{t}_{1},\hat{t}_{2},\hat{s}_{1},\hat{s}_{2},\hat{x},\hat{y})
\\
  &\geq w(\hat{t}_{1},\hat{t}_{2},\hat{x})-\tilde{w}(\hat{t}_{1},\hat{t}_{2},\hat{x})
  -\rho_{\epsilon,\lambda_{1},\lambda_{2}}(\hat{t}_{1},\hat{t}_{2},\hat{t}_{1},\hat{t}_{2},\hat{x},\hat{x}).
\end{aligned}
$$
Therefore, we obtain
$$
  \begin{aligned}
  \epsilon^{-2}((\hat{t}_{1}-\hat{s}_{1})^{2}&+(\hat{t}_{2}-\hat{s}_{2})^{2}+\left\|\hat{x}-\hat{y}\right\|^{2})
  \leq
   w(\hat{x},\hat{t}_{1},\hat{t}_{2})-\tilde{w}(\hat{y},\hat{s}_{1},\hat{s}_{2})
\\
   &-\frac{\lambda_{1}}{2}(\hat{t}_{1}-\hat{s}_{1})+\frac{\lambda_{2}}{2}(\hat{t}_{2}-\hat{s}_{2})+\epsilon(\hat{x}-\hat{y})(\hat{x}+\hat{y}).
   \end{aligned}
$$
Then, by \eqref{x1}, \eqref{x2} and the Lipschitz continuity of
$\tilde{w}$, we have
\begin{equation}
\label{x3}
|\hat{t}_{1}-\hat{s}_{1}|,|\hat{t}_{2}-\hat{s}_{2}|,\left\|\hat{x}-\hat{y}\right\|=o(\epsilon).
\end{equation}
Now, let $\omega$ be the modulus of continuity of $w$; that is,
$$|w(t_{1},t_{2},x)-\tilde{w}(s_{1},s_{2},y)|\leq \omega(|t_{1}-s_{1}|+|t_{2}-s_{2}|+\left\|x-y\right\|)$$
for all $x,y\in\R^{n}$, $0\leq t,~s\leq T$, and
$\omega(r)\rightarrow 0$ as $r\rightarrow 0$. Similarly,
$\tilde{\omega}(\cdot)$ will denote the modulus of continuity of
$\tilde{w}$. Then ($\ref{alp}$) implies

\begin{equation*}
\begin{aligned}
   \frac{\alpha}{2}\leq w(\hat{t}_{1},\hat{t}_{2},\hat{x})-\tilde{w}(\hat{s}_{1},\hat{s}_{2},\hat{y})&=[w(\hat{t}_{1},\hat{t}_{2},\hat{x})-w(\hat{t}_{1},0,\hat{x})]+[w(\hat{t}_{1},0,\hat{x})-w(0,0,\hat{x})]
   \\
   &+[w(0,0,\hat{x})-\tilde{w}(0,0,\hat{x})]+[\tilde{w}(0,0,\hat{x})-\tilde{w}(\hat{t}_{1},0,\hat{x})]
   \\
   &+[\tilde{w}(\hat{t}_{1},0,\hat{x})-\tilde{w}(\hat{t}_{1},\hat{t}_{2},\hat{x})]+[\tilde{w}(\hat{t}_{1},\hat{t}_{2},\hat{x})-\tilde{w}(\hat{s}_{1},\hat{s}_{2},\hat{y})].
\end{aligned}
\end{equation*}
Therefore, using ($\ref{x1}$), ($\ref{x3}$) and the initial
condition, we have

$$\frac{\alpha}{2}\leq \omega(\hat{t}_{2})+\omega(\hat{t}_{1})+\tilde{\omega}(\hat{t}_{1})+\tilde{\omega}(\hat{t_{2}})+\tilde{\omega}(o(\epsilon)).$$
As $\epsilon$ is a positive arbitrary number, we can take it so
small as necessary to obtain

$$\frac{\alpha}{4}\leq \omega(\hat{t}_{2})+\omega(\hat{t}_{1})+\tilde{\omega}(\hat{t}_{1})+\tilde{\omega}(\hat{t_{2}})$$
and this implies for some constant $\mu>0$,
$$
  \hat{t}_{1},\hat{t}_{2}\geq\mu>0.
$$
Analogously, we have $\hat{s}_{1},\hat{s}_{2}\geq\mu>0$.
\end{remark}
%

\section{Multi-time conservation laws} \label{MTCL1D}

Once we have establish existence and uniqueness for multi-time
Hamilton-Jacobi system, we are going to use it in this section, in
order to show solvability of the multi-time system of conservation
laws. Therefore, we fixe $d,s$ equals one and for given $H_i$
(i=1,2) two smooth (uniformly) convex flux-function, we consider
the following Cauchy problem: Find $u: (0,\infty)^2 \times \R \to
\R$, satisfying
\begin{equation}
\label{MTSCL1D}
   \begin{aligned}
      u_{t_1} + \partial_x H_1(u) &= 0 \qquad \text{in $(0,\infty)^2 \times
      \R$}, \\[5pt]
      u_{t_2} + \partial_x H_2(u) &= 0 \qquad \text{in $(0,\infty)^2 \times
      \R$}, \\[5pt]
      u(0,0,x) &= u_0(x) \quad \text{on $\R$},
   \end{aligned}
\end{equation}
where $u_0 \in L^\infty(\R)$ is a given initial-data. With no loss
of generality, we assume $H_i(0)= 0$ $(i=1,2)$. Following a usual
strategy to $1D$ scalar conservation laws, we define
\begin{equation}
\label{GU0}
  g(x):= \int_0^x u_0(y) \, dy \qquad (x \in \R),
\end{equation}
thus $g$ is a Lipschitz function with $\Lip(g)= \|u_0\|_\infty$,
and recall the multi-time Lax formula given by \eqref{LFMHJE}.
Thus by Theorem \ref{GHLFET}, $w$ solves the multi-time
Hamilton-Jacobi system \eqref{HJS} and, if we assume that $w$ is
smooth, then we can differentiate that system with respect to $x$,
to deduce
\begin{equation}
\label{DHJS}
   \begin{aligned}
      w_{x \, t_1} + \partial_x H_1(w_x) &= 0 \qquad \text{in $(0,\infty)^2 \times
      \R$}, \\[5pt]
      w_{x \, t_2} + \partial_x H_2(w_x) &= 0 \qquad \text{in $(0,\infty)^2 \times
      \R$}, \\[5pt]
      w_x(0,0,x) &= u_0(x) \quad \text{on $\R$}.
   \end{aligned}
\end{equation}
Now, setting $u= w_x$ we obtain that $u$ solves the system
\eqref{MTSCL1D}. Certainly, the computation is only formal,
indeed, even that the function $w$ is differentiable a.e., we are
not allowed to differentiate $H_1(w_x)$ with respect to $x$,
similarly to $H_2$. Although,
\begin{equation}
\begin{aligned}
\label{SCL} u(t_1,t_2,x):&= \partial_x \big( \min_{y \in \R}
 \big\{ \big( t_1 H_1 + t_2 H_2\big)^*(x-y) + g(y) \big\} \big)
 \\
 &= \partial_x \big( (\bold{t} \cdot \bold{H})^* \ic g \big)(x),
\end{aligned}
\end{equation}
seems to be the best candidate for a solution of the Cauchy
problem \eqref{MTSCL1D}. In fact, we will show that such function
$u$ as defined above is a (weak integral) solution, but before,
let's us first show a more useful formula.

\begin{lemma}
\label{MTLOF} $(${\bf Multi-time Lax-Oleinik formula}$)$. Assume
$H_i:\R \to \R$ $(i=1,2)$ are smooth uniformly convex, $u_0 \in
L^\infty(\R)$ and $g$ is given by \eqref{GU0}. Then, for each
$t_1, t_2
> 0$, there exists for all but at most countably many values $x
\in \R$, such that \eqref{SCL} has the following form
\begin{equation}
\label{SCLLO} u(t_1,t_2,x)= \Big((t_1 H_1)^* \ic (t_2 H_2)^*
\Big)'\big(x-y(t_1,t_2,x)\big),
\end{equation}
where the mapping $x \mapsto y(t_1,t_2,x)$ is nondecreasing.
Moreover, for each $z>0$
\begin{equation}
\label{SCLLOF} u(t_1,t_2,x+z) - u(t_1,t_2,x) \leq
\Lip\big(\big((t_1 H_1 + t_2 H_2)^*\big)'\big) \; z.
\end{equation}
\end{lemma}

\begin{definition}
\label{LOF} Equation \eqref{SCLLO} is called the multi-time
Lax-Oleinik formula.
\end{definition}

\begin{proof}
1. Fix $t_1,t_2>0$, $x_1<x_2$. There exists at least one point
$y_1 \in \R$, such that
\begin{equation}
\label{SCLLO0}
  w(t_1,t_2,x_1)= \big(t_1 H_1 + t_2 H_2 \big)^*(x_1-y_1) + g(y_1).
\end{equation}
Now, we claim that, for each $y < y_1$,
$$
  \big(t_1 H_1 + t_2 H_2 \big)^*(x_2-y_1) + g(y_1)
  <
  \big(t_1 H_1 + t_2 H_2 \big)^*(x_2 - y) + g(y).
$$
Indeed, let $\tau \in (0,1)$, given by
$$
  \tau= \frac{y_1 -y}{(x_2 -x_1) + (y_1 - y)},
$$
and for convenience, we write
$$
  \begin{aligned}
  x_2 - y_1&= \tau \, (x_1 - y_1) + (1-\tau) \, (x_2 - y), \\
 x_1 - y&= (1-\tau) \, (x_1 - y_1) + \tau \, (x_2 - y).
  \end{aligned}
$$
Therefore, since $(H_i^*)''>0$ $(i=1,2)$, it follows that
$$
  \begin{aligned}
\big(\bold{t} \cdot \bold{H} \big)^*(x_2 - y_1)&< \tau \,
\big(\bold{t} \cdot \bold{H} \big)^*(x_1 - y_1)
  + (1-\tau) \, \big(\bold{t} \cdot \bold{H} \big)^*(x_2 - y), \\
 \big(\bold{t} \cdot \bold{H} \big)^*(x_1 - y)&<  (1-\tau)
 \, \big(\bold{t} \cdot \bold{H} \big)^*(x_1 - y_1)
 + \tau \, \big(\bold{t} \cdot \bold{H} \big)^*(x_2 - y).
  \end{aligned}
$$
Then, combining the two above inequalities, we obtain
\begin{equation}
\label{SCLLO1}
\begin{aligned}
   \big(\bold{t} \cdot \bold{H} \big)^*(x_2 - y_1) &+
   \big(\bold{t} \cdot \bold{H} \big)^*(x_1 - y)
   \\
   &< \big(\bold{t} \cdot \bold{H} \big)^*(x_1 - y_1)
  + \big(\bold{t} \cdot \bold{H} \big)^*(x_2 - y).
\end{aligned}
\end{equation}
Moreover, by the definition of $w(t_1,t_2,x_1)$, we have
\begin{equation}
\label{SCLLO2}
\begin{aligned}
  -\big(t_1 H_1 + t_2 H_2 \big)^*(x_1 - y) - g(y)
   \leq -\big(t_1 H_1 + t_2 H_2 \big)^*(x_1 - y_1) - g(y_1).
\end{aligned}
\end{equation}
Then, from \eqref{SCLLO1} and \eqref{SCLLO2}
$$
  \big(t_1 H_1 + t_2 H_2 \big)^*(x_2 - y_1) + g(y_1)
   < \big(t_1 H_1 + t_2 H_2 \big)^*(x_2 - y) + g(y),
$$
and so the claim is proved.

\medskip
2. From the claim proved before, we observe that to compute the
minimum below, i.e.
$$
  \min_{y \in \R}
 \left\{
 \big(t_1 H_1 + t_2 H_2 \big)^*(x_2-y)
 + g(y) \right\},
$$
we only need to consider those $y \geq y_1$, where $y_1$ satisfies
\eqref{SCLLO0}. Therefore, for each $t_1,t_2>0$ and $x \in \R$, we
could define the point $y(t_1,t_2,x)$ equal to the smallest value
of those points $y$ giving the minimum of
$$
 \big(t_1 H_1 + t_2 H_2 \big)^*(x-y)
 + g(y).
$$
Consequently, for each $t_1,t_2>0$, the mapping $x \mapsto
y(t_1,t_2,x)$ is nondecreasing, thus continuous for all but at
most countably many $x \in \R$. Moreover, at o such point $x$, the
value $y(t_1,t_2,x)$ is the unique those $y$ yielding the minimum.

\medskip
3. Since the function $w$ is Lipschitz, thus differentiable a.e.
and the mapping $x \mapsto y(t_1,t_2,x)$ is monotone and so
differentiable a.e. as well, given $t_1,t_2>0$, for a.e. $x \in
\R$, the mappings
$$
  \begin{aligned}
  x &\mapsto \big(t_1 H_1 + t_2 H_2\big)^*(x - y(t_1,t_2,x)),
  \\
  x &\mapsto g(y(t_1,t_2,x))
  \end{aligned}
$$
are also differentiable for a.e. $x \in \R$. Then, we have for
such a differentiable point $x$,
$$
  \begin{aligned}
  u(t_1,t_2,x)&= \partial_x \Big( \big(t_1 H_1 + t_2 H_2 \big)^*
  (x-y(t_1,t_2,x)) + g(y) \Big)
 \\
 &= \Big( \big( t_1 H_1 + t_2 H_2 \big)^* \Big)'
 (x-y(t_1,t_2,x))
 \; \Big(1 - y_x(t_1,t_2,x)\Big)
 \\
  &+ \partial_x\Big(g(y(t_1,t_2,x))\Big).
  \end{aligned}
$$
But since the mapping $y \mapsto (t_1 H_1 + t_2 H_2)^* + g$ has a
minimum at $y= y(t_1,t_2,x)$, it follows that
$$
 - \Big(\big(t_1 H_1 + t_2 H_2 \big)^*\Big)'(x-y(t_1,t_2,x)) \;
 y_x(t_1,t_2,x)
 + \partial_x\big(g(y(t_1,t_2,x))\big)= 0,
$$
and thus we obtain \eqref{SCLLO}.

\medskip
4. Finally, by equation \eqref{SCLLO}, the monotonicity of
$\big((t_1 H_1 + t_2 H_2)^*\big)'$ and $y(t_1,t_2,\cdot)$ as well,
we have for each $z>0$
$$
  \begin{aligned}
  u(t_1,t_2,x)&= \big((t_1 H_1 + t_2 H_2)^*\big)'(x-y(t_1,t_2,x))
 \\
 &\geq \big((t_1 H_1 + t_2 H_2)^*\big)'(x-y(t_1,t_2,x+z))
 \\
 &\geq \big((t_1 H_1 + t_2 H_2)^*\big)'(x+z-y(t_1,t_2,x+z))
 \\
 &\quad -
 \Lip\big(\big((t_1 H_1 + t_2 H_2)^*\big)'\big) \; z
 \\
 &= u(t_1,t_2,x+z) -  \Lip\big(\big((t_1 H_1 + t_2 H_2)^*\big)'\big) \; z.
  \end{aligned}
$$
Therefore, we obtain
$$
  u(t_1,t_2,x+z) - u(t_1,t_2,x)\leq \Lip\big(\big((t_1 H_1 + t_2 H_2)^*\big)'\big) \; z.
$$
%
\end{proof}

\bigskip
\subsection{Existence}
Now we are ready to show the solvability of the multi-time system
of conservation laws in $1D$ for two independent times. First, let
us define in which sense a bounded and measurable real function
$u$ defined in $(0,\infty)^2 \times \R$ is a weak (integral)
solution of \eqref{MTSCL1D}.

\begin{definition}
\label{SDMTCL1D} Given $u_0 \in L^\infty(\R)$, a function $u \in
L^\infty((0,\infty)^2 \times \R)$ is said a weak integral solution
of the Cauchy problem \eqref{MTSCL1D}, if it satisfies

$\bullet$ Multi-time conservation laws: For all $\varphi \in
C_0^\infty((0,\infty)^2 \times \R)$
\begin{equation}
 \int_0^\infty \!\!\!\!\int_0^\infty \!\!\!\!\int_\R
 \big( u \, \varphi_{t_1} + H_1(u) \, \varphi_x \big) \, dx dt_1 dt_2= 0,
\end{equation}
\begin{equation}
 \int_0^\infty \!\!\!\!\int_0^\infty \!\!\!\!\int_\R
 \big( u \, \varphi_{t_2} + H_2(u) \, \varphi_x \big) \, dx dt_1 dt_2=
 0.
\end{equation}

$\bullet$ Initial condition: For any $\gamma \in L^1(\R)$
\begin{equation}
\label{ICCL} \ess \!\!\!\!\!\! \lim_{t_1,t_2 \to 0^+} \int_\R
\big( u(t_1,t_2,x) - u_0(x) \big) \, \gamma(x) \, dx= 0.
\end{equation}
\end{definition}

\bigskip
\begin{theorem}
\label{TMTCLCP} The function $u \in L^\infty((0,\infty)^2 \times
\R)$ given by Lemma \ref{MTLOF}, equation \eqref{SCLLO}, is a weak
solution of the Cauchy problem \eqref{MTSCL1D}.
\end{theorem}

\begin{proof}
First, we define for $t_1,t_2>0$ and $x \in \R$,
$$
  w(t_1,t_2,x)= \min_{y \in \R}
 \left\{
 (t_1 H_1 + t_2 H_2 )^*(x-y)
 + g(y) \right\},
$$
which by Theorem \ref{GHLFET} is a Lipschitz continuous function,
differentiable a.e in $(0,\infty)^2 \times \R$, and solves
\begin{equation}
\label{PMTCL}
  \begin{aligned}
      w_{t_1} + H_1(w_x) &= 0 \qquad \text{a.e. in $(0,\infty)^2 \times
      \R$}, \\[5pt]
      w_{t_2} + H_2(w_x) &= 0 \qquad \text{a.e. in $(0,\infty)^2 \times
      \R$}, \\[5pt]
      w(0,0,x) &= g(x) \quad \text{on $\R$}.
   \end{aligned}
\end{equation}

Now, we take $\varphi \in C_0^\infty ((0,\infty)^2 \times \R)$
multiply the first equation in \eqref{PMTCL} by $\varphi_x$ and
integrate over $(0,\infty)^2 \times \R$, to obtain
$$
   \int_0^\infty \!\!\!\!\int_0^\infty \!\!\!\!\int_\R
 \big( w_{t_1} \, \varphi_x + H_1(w_x) \, \varphi_x \big) \, dx dt_1 dt_2=
 0.
$$
Then, we observe that
$$
  \begin{aligned}
   \int_0^\infty \!\!\!\!\int_0^\infty \!\!\!\!\int_\R
   w_{t_1} \, \varphi_x \, dx dt_1 dt_2
   &= -
  \int_0^\infty \!\!\!\!\int_0^\infty \!\!\!\!\int_\R
   w \, \varphi_{t_1 x} \, dx dt_1 dt_2
   \\[5pt]
   &=
   \int_0^\infty \!\!\!\!\int_0^\infty \!\!\!\!\int_\R
   w_{x} \, \varphi_{t_1} \, dx dt_1 dt_2,
  \end{aligned}
$$
where we are allowed to integrate by parts, since the mapping $x
\mapsto w(t_1,t_2,x)$ is Lipschitz continuous and then, absolutely
continuous for each $t_1,t_2>0$. Moreover, for each $t_2>0$ and $x
\in \R$, the mapping $t_1 \mapsto w(t_1,t_2,x)$ is also absolutely
continuous. Therefore, we have
$$
 \int_0^\infty \!\!\!\!\int_0^\infty \!\!\!\!\int_\R
 \big( w_x \, \varphi_{t_1} + H_1(w_x) \, \varphi_x \big) \, dx dt_1 dt_2= 0,
$$
and by similarly argument, we obtain
$$
 \int_0^\infty \!\!\!\!\int_0^\infty \!\!\!\!\int_\R
 \big( w_x \, \varphi_{t_2} + H_2(w_x) \, \varphi_x \big) \, dx dt_1 dt_2=
 0.
$$
Finally, we recall that $u= w_x$ $a.e.$ as precisely defined by
\eqref{SCLLO}. Then, the Multi-time conservation laws condition at
Definition \ref{SDMTCL1D} is satisfied.

 To show the initial-condition, we apply the same strategy before and the result follows using \eqref{IDMTHJE}.
\end{proof}

\bigskip
\subsection{Uniqueness}
We show the existence of a weak integral solution $u$ to the
problem \eqref{MTSCL1D}, where $u$ is given by \eqref{SCLLO}.
Recall that, the integral solution is slight different from the
entropy solution given by Definition \ref{ENTROPYSOLUTION}, that
is, a measurable and bounded function $u(t_1,t_2,x)$ is an entropy
solution to \eqref{MTSCL1D}, if for all entropy pair
$(\eta(u),q_i(u))$ $(i=1,2)$, and for each $T>0$, the following
holds true
\begin{equation}
 \int_0^T \!\!\!\!\int_0^T \!\!\!\!\int_\R
 \big( \eta(u) \, \varphi_{t_1} + q_1(u) \, \varphi_x \big) \, dx dt_1 dt_2 \geq 0,
\end{equation}
\begin{equation}
 \int_0^T \!\!\!\!\int_0^T \!\!\!\!\int_\R
 \big( \eta(u) \, \varphi_{t_2} + q_2(u) \, \varphi_x \big) \, dx dt_1
 dt_2\geq 0,
\end{equation}
for each non-negative test function $\varphi \in
C_0^\infty((0,T)^2 \times \R)$, and also the initial-condition
\eqref{ICCL} is satisfied. It follows by \eqref{SCLLOF} that, for
each $t_1,t_2 \in (0,T)$ fixed, $u(t_1,t_2,\cdot)$ has locally
bounded variation. Indeed, we know that for each $z>0$
$$
  \frac{u(t_1,t_2,x+z) - u(t_1,t_2,x)}{z} \leq c,
$$
where $c:= \Lip\big(\big((t_1 H_1 + t_2 H_2)^*\big)'\big)$. Let,
$\td{u}(t_1,t_2,x)= u(t_1,t_2,x) - \td{c} \, x$, for $\td{c} > c$.
Then, we have for each $z>0$
$$
  \td{u}(t_1,t_2,x+z)-\td{u}(t_1,t_2,x) < 0,
$$
that is, $\td{u}(t_1,t_2,\cdot)$ is a decreasing function and
hence has locally bounded total variation. Since this is also true
for $\td{c} \, x$, we obtain that $u(t_1,t_2,\cdot)$ has locally
bounded variation.
Therefore, the well-known theory of Vol'pert \cite{AIV} allow us
to apply the chain rule for $BV$ functions, and write for a.e. $x
\in \R$, $i=1,2$
$$
  \partial_x H_i(u(t_1,t_2,x))= H'_i(u(t_1,t_2,x)) \;
  (u(t_1,t_2,x))_x,
$$
and thus since $u$ is an integral solution, we have in the sense
of measures
\begin{equation}
\label{MTSCL1DSD}
      |u_{t_i}| \leq \max_{\xi \in B_{\|u\|_\infty}(0)} |H'_i(\xi)| \;\;
      |u_x|,
\end{equation}
that is to say, $u_{t_1}, u_{t_2}$ are locally Radon measures.

\bigskip
Now, let $\eta$ be a smooth convex function. Again with no loss of
generality we may as well also take $\eta(0)= 0$. Then, we
multiply \eqref{MTSCL1DSD} by $\eta'(u)$, and apply again the
chain rule for BV functions to obtain in the measure sense
\begin{equation}
\label{MTSCL1DES}
   \begin{aligned}
      \eta(u)_{t_1} + \partial_x q_1(u) &= 0, \\[5pt]
      \eta(u)_{t_2} + \partial_x q_2(u) &= 0.
   \end{aligned}
\end{equation}
Consequently, it is not difficult to see that, the integral
solution $u$ is in fact an entropy solution, where it is crucial
the estimate \eqref{MTSCL1DSD} in order to show the initial data
\eqref{idcles}. Moreover, by a standard approximation procedure,
we may assume that the pair $(\eta,q_i)$ $(i=1,2)$ are the Kruzkov
entropies, that is,
\begin{equation}
 \int_0^T \!\!\!\!\int_0^T \!\!\!\!\int_\R
 \big( |u-v| \, \varphi_{t_1} + \sgn(u-v) \big(H_1(u)-H_1(v)\big)
  \, \varphi_x \big) \, dx dt_1 dt_2= 0,
\end{equation}
\begin{equation}
 \int_0^T \!\!\!\!\int_0^T \!\!\!\!\int_\R
\big( |u-v| \, \varphi_{t_2} + \sgn(u-v) \big(H_2(u)-H_2(v)\big)
  \, \varphi_x \big) \, dx dt_1 dt_2= 0,
\end{equation}
for each $v \in \R$ fixed and all test function $\varphi \in
C_0^\infty((0,T)^2 \times \R)$.
Therefore, we are in position to apply the doubling variables
technic due to Kruzkov, see \cite{KRUZKOV}. In fact, this is
nowadays a standard procedure, thus adapted to our case leads to
the following result

\begin{lemma}
Let $u$ and $v$ be two entropy solutions to the problem
\eqref{MTSCL1D} corresponding to initial data $u_0$, $v_0$
respectively. Then, we have the $L^1$-contraction type
inequalities
\begin{equation}
\label{CONTRAC}
  \begin{aligned}
 \int_0^T \!\!\!\!\int_{B_R(0)}
 |u(t_1,\tau,x)-&v(t_1,\tau,x)| \; \zeta_2(\tau) \; dx d\tau
 \\
   & \qquad \leq \int_0^T \!\!\!\!\int_{B_{R_1}(0)}
|u(0,\tau,x)-v(0,\tau,x)| \; \zeta_2(\tau) \;dxd\tau,
\\[5pt]
 \int_0^T \!\!\!\!\int_{B_R(0)}
 |u(\tau,t_2,x)-&v(\tau,t_2,x)| \; \zeta_1(\tau) \;dxd\tau
\\
  & \qquad \leq \int_0^T \!\!\!\!\int_{B_{R_2}(0)}
|u(\tau,0,x)-v(\tau,0,x)| \; \zeta_1(\tau) \;dxd\tau,
  \end{aligned}
\end{equation}
which holds for all ball $B_R(0)$, $R>0$ and almost all $t_1,t_2
>0$, where for $i=1,2$, $\zeta_i \in C_0^\infty(0,T)$,
$B_{R_i}= B_{R+M_i t_i}(0)$, and $M_i$ denotes the Lipschitz
constant of $H_i$.
\end{lemma}

\begin{theorem}
\label{UNQTCL} Let $u$ and $v$ be two entropy solutions to the
problem \eqref{MTSCL1D} corresponding to initial data $u_0$, $v_0$
respectively. If $u_0= v_0$ almost everywhere, then $u= v$ almost
everywhere.
\end{theorem}

\begin{proof}
For $\delta >0$, we take $\zeta_1(\tau)= \chi_{(0,\delta)}(\tau)$
in the second inequality of \eqref{CONTRAC}. Then, dividing by
$\delta$ both sides of the inequality, and passing to the limit as
$\delta \to 0^+$, we obtain
$$
  \int_\R
  |u(0,t_2,x)-v(0,t_2,x)| \; dx= 0.
$$
Similarly, for $\theta>0$ sufficiently small, we take
$\zeta_2(\tau)= \chi_{(t_2-\theta,t_2+\theta)}(\tau)$ in the first
inequality of \eqref{CONTRAC}. Again dividing the inequality by
$\theta$ and passing to limit as $\theta$ goes to $0^+$, the
uniqueness result follows, that is, $u \equiv v$ almost
everywhere.
\end{proof}

\section*{Acknowledgements}
Aldo Bazan is supported by FAPERJ by the grant 2009.2848.0.
Wladimir Neves is partially supported by FAPERJ through the grant
E-26/ 111.564/2008 entitled {\sl ``Analysis, Geometry and
Applications''} and by
Pronex-FAPERJ through the grant E-26/ 110.560/2010 entitled \textsl{%
"Nonlinear Partial Differential Equations"}.


%

\end{document}